\documentclass[11pt]{article}
\usepackage[utf8]{inputenc}
\usepackage[T1]{fontenc}
\usepackage{amsmath,amsfonts,amssymb,amsthm,mathtools}
\usepackage{mathrsfs}
\usepackage{verbatim}
\usepackage[dvipsnames]{xcolor}
\usepackage[margin=1in]{geometry}
\usepackage{microtype}
\usepackage[colorlinks=true,linkcolor=blue,citecolor=blue,urlcolor=blue]{hyperref}

\theoremstyle{plain}
\newtheorem{theorem}{Theorem}[section]
\newtheorem{assumption}[theorem]{Assumption}

\newtheorem{lemma}[theorem]{Lemma}
\newtheorem{proposition}[theorem]{Proposition}
\theoremstyle{remark}

\newtheorem{remark}[theorem]{Remark}
\numberwithin{equation}{section}

\title{Gaussian equivalence for CNN with non-Gaussian weights}
\author{Federico Bassetti\thanks{Dipartimento di Matematica, Politecnico di
    Milano, Milano, Italy. \texttt{federico.bassetti@polimi.it}}
  \and
  Vassili De Palma\thanks{Departement Informatik, ETH Z\"urich,
    Switzerland. \texttt{vdepalma@ethz.ch}}}
\date{}

\newcommand{\RE}{\mathbb{R}}

\newcommand{\mf}{\mathfrak{m}}

\newcommand{\CC}{\mathcal{C}}

\newcommand{\bx}{\mathbf{x}}
\newcommand{\by}{\mathbf{y}}

\newcommand{\CF}{\mathcal{F}}

\newcommand{\CH}{\mathcal{H}}

\newcommand{\PP}{\mathbb{P}}

\newcommand{\CN}{\mathcal{N}}

\newcommand{\CM}{\mathcal{M}}

\newcommand{\CX}{\mathcal{X}}

\newcommand{\E}{\mathbf{E}}
\newcommand{\Var}{\mathbf{Var}}

\DeclareMathOperator{\Cov}{Cov}

\newcommand{\CU}{\mathcal{U}} %

\newcommand{\CL}{\mathcal{L}}
\newcommand{\CK}{\mathcal{K}}

\begin{document}

\maketitle

\begin{abstract}
We study the infinite-width limit of convolutional neural networks
  with random weights and show that the Gaussian process limit does not
  require Gaussian priors. Assuming that the weights are i.i.d.,
  centered, and have sufficiently many finite moments,  converges weakly
in the space of continuous functions on the compact input set
to independent copies of a centered Gaussian process.
 Our  framework accommodates activation functions of polynomial growth together
  with a broad class of patch-extraction operators, covering standard
  convolutions, padding, stride, and common pooling mechanisms.
\end{abstract}

\noindent\textbf{Keywords:} convolutional neural network; Gaussian process;
infinite-width limit; central limit theorem; non-Gaussian weights.

\noindent\textbf{MSC2020:} Primary 60F05; 60G15. Secondary 68T07.

%\section{Introduction}

The infinite-width limit of neural networks has attracted considerable
attention in recent years, both in probability and machine learning.
Under suitable random initialization schemes, the pre-activations of
fully connected networks converge to Gaussian processes, giving rise to
the so-called Neural Network Gaussian Process (NNGP), see 
\cite{g.2018gaussian,lee2018deep,Hanin2023}.
Beyond weak convergence, recent works have developed quantitative
refinements of this limit
\cite{Celli2026Wide,favaro2023quantitative,CelliPeccati2025Entropic},
as well as moderate and large deviation principles for both the network
output and the conditional covariance function
\cite{ABH2026,MacciTrevisanetal,MaPaTo24,Vog24}.
Other asymptotic regimes, where the number of layers or the number of
observations also diverges, have likewise been investigated
\cite{aiudi2023,baglioni2024predictive,BaLaRo,hanin2024,Mufan2023}.

More recently, analogous Gaussian equivalence results have been obtained
for convolutional neural networks (CNNs) with Gaussian weights, where the
convolutional architecture and the presence of pooling substantially
complicate the analysis
\cite{Rasmussen2019,novak2019bayesian,Yang2019TensorI,
yang2020feature,BassettiDePalmaLadelli2026LDP,BassettietalJMLR}. Finite dimensional limits for networks with non-Gaussian weights have
also been established in \cite{YangGolikov},
within the general framework of tensor programs.
The purpose of this note is to establish functional Gaussian equivalence for CNNs
under finite moment assumptions on the weights.

We consider a general class of convolutional
networks with $L$ hidden layers, described through deterministic
patch-extraction operators of polynomial growth. These operators
cover nonlinear activations, padding, stride and pooling, including
average and max pooling, and their composition yields standard CNN
architectures. Explicit examples are described in
\cite{BassettiDePalmaLadelli2026LDP}.

The weights are independent, centered and identically distributed within
each layer, with no restriction on their law beyond the existence of finitely
many moments, and we let the number of channels of every hidden layer tend to
infinity.

Our main result, Theorem~\ref{thm:cnn_clt}, states that, at every layer
beyond the first, any fixed number of channels converges weakly
in the space of continuous functions on the compact input set
to independent copies of a centered Gaussian process.
The limiting covariance kernels are defined recursively and depend
on the weight distributions through their variances and the
distribution of the first-layer pre-activations.

The proof proceeds by induction over the layers. Moment and increment bounds
are propagated through the network by means of Rosenthal's inequality.
Finite-dimensional convergence follows from a conditional Lyapunov central
limit theorem, combined with a law of large numbers for exchangeable
triangular arrays that identifies the limiting conditional variances.
Tightness in the space of continuous functions is then obtained from the
increment bounds via a Kolmogorov--Chentsov type criterion.

The remainder of the paper is organized as follows. In Section~\ref{sec:setup}
we describe the convolutional architecture, state our assumptions, and present
our main results. Section~\ref{sec:proofs} is devoted to the proofs.

\section{Convolutional neural networks and Gaussian equivalence}
\label{sec:setup}

We consider a convolutional neural network with $L$ hidden layers.
At layer $\ell$, neurons are indexed by a channel index
$c=1,\dots,C_\ell$ and a spatial index
$i\in\Lambda_\ell$. For simplicity, and without loss of generality,
we identify the spatial domain with
$\Lambda_\ell=\{1,\dots,N_\ell\}$.

Throughout the paper,
$C_\ell$ denotes the number of channels,
$N_\ell$ the spatial dimension,
$\CM_\ell$ the receptive field (mask) associated with layer $\ell$,
and $M_\ell:=|\CM_\ell|$ its cardinality.

For every neuron $i=1,\dots,N_{\ell+1}$, let
$T^{(i,\ell)}:\RE^{N_\ell}\rightarrow\RE^{M_\ell}$
be the corresponding patch-extraction operator.
This operator extracts the receptive field centered at location $i$
and may include nonlinear activation, pooling, padding, stride,
or other deterministic transformations.
We write
$T^{(i,\ell)}(z) = \bigl[ T^{(i,\ell)}_m(z): m\in\CM_\ell \bigr]$.

For an input
$\bx=[x_{c,i}] \in \RE^{C_0\times N_0}$,
the first hidden layer is defined by
\begin{equation}
\begin{split}
h_{c,i}^{(1)}(\bx)
=
\frac{1}{\sqrt{M_0C_0}}
\sum_{c'=1}^{C_0}
\sum_{m\in\CM_0}
W^{(0)}_{m,c',c}
\,T^{(i,0)}_m(x_{c',:}),
\\
c=1,\dots,C_1,
\qquad
i=1,\dots,N_1,
\end{split}
\end{equation}
where
$x_{c,:} = (x_{c,1},\dots,x_{c,N_0})^\top$.
For layers $\ell=1,\dots,L$, define
\begin{equation}\label{convolutiona.recur1}
\begin{split}
h_{c,i}^{(\ell+1)}(\bx)
=
\frac{1}{\sqrt{M_\ell C_\ell}}
\sum_{c'=1}^{C_\ell}
\sum_{m\in\CM_\ell}
W^{(\ell)}_{m,c',c}
\,T^{(i,\ell)}_m
\bigl(
h^{(\ell)}_{c',:}(\bx)
\bigr),
\\
c=1,\dots,C_{\ell+1},
\qquad
i=1,\dots,N_{\ell+1}.
\end{split}
\end{equation}
where
$h^{(\ell)}_{c,:}(\bx) = (h^{(\ell)}_{c,1}(\bx),\dots,h^{(\ell)}_{c,N_\ell}(\bx))^\top$.
 Examples of standard CNN architectures covered by this formalism, including
different pooling and padding schemes, can be found in \cite{BassettiDePalmaLadelli2026LDP}.
For simplicity we assume zero biases. 

In the rest of the paper, 
$\CX\subset\RE^{C_0\times N_0}$ is a fixed compact set of
inputs  and $\|\cdot\|$ is the Euclidean norm.

The following assumptions will be used throughout the paper.

\begin{assumption}[Infinite channels limit]\label{A1CNN-bis}
$L$, $N_0,\dots,N_{L+1}$,
$C_0$ and $C_{L+1}$ are fixed, whereas $C_\ell=C_\ell(n)$ with 
\[
\lim_{n \to +\infty} C_\ell(n)=+\infty, \quad \ell=1,\dots,L.
\]
\end{assumption}

\begin{assumption}[Local Lipschitz continuity]
\label{A2_0CNN-tris}
The operators $T^{(i,\ell)}$ are continuous and there exist $r\ge1$ and
$A>0$ such that, for every $i,\ell,m$ and all $z,z'\in\RE^{N_\ell}$,
\begin{align}
%\label{eq:polyT}
%\bigl|T^{(i,\ell)}_m(z)\bigr|
%&\le A\bigl(1+\|z\|^{r}\bigr),
%\\[2pt]
\label{eq:lipT}
\bigl|T^{(i,\ell)}_m(z)-T^{(i,\ell)}_m(z')\bigr|
\le A\bigl(1+\|z\|^{r-1}+\|z'\|^{r-1}\bigr)\|z-z'\| .
\end{align}
\end{assumption}

Condition \eqref{eq:lipT}  holds in particular for
compositions of ReLU or smooth sigmoidal activations with average or max
pooling, padding and stride.

The network is randomly initialized through a collection of random weights
$\{W^{(\ell)}_{m,c',c}\}$.
Their distributional assumptions are stated below.

\begin{assumption}[Weights]
\label{A3nongaussian}
Within each layer the weights $W^{(\ell)}_{m,c',c}$ are i.i.d., weights of
different layers are independent, $\E[W^{(\ell)}_{m,c',c}]=0$,
$\Var(W^{(\ell)}_{m,c',c})=\mf^{(\ell)}_2<+\infty$, and there exists
$q>2$
such that 
\begin{equation}
\label{eq:moments}
\E\bigl|W^{(\ell)}_{m,c',c}\bigr|^{q_{\ell}}<+\infty,
\qquad \ell=0,\dots,L,
\end{equation}
where the exponents $q_\ell$ are defined backwards from the top layer
by
\begin{equation}
\label{eq:aexp}
q_{L}:=q,
\qquad
q_{\ell-1}=r q_{\ell}
%a_\ell:=r_0\,a_{\ell+1}
\quad(\ell=L,\dots,1),
\qquad\text{that is}\quad
q_\ell=q \,r^{L-\ell} .
\end{equation}
\end{assumption}

\subsection{Gaussian equivalence at the functional level}

Under Assumption~\ref{A1CNN-bis}, the activations implicitly depend on $n$
through the channel widths $C_\ell=C_\ell(n)$. 
For every layer
$\ell=1,\dots,L+1$ and every $1\le C\le C_\ell$, define
\[
\CH_{1:C}^{(\ell,n)}(\bx)
:=
\Bigl(
h^{(\ell)}_{c,i}(\bx):
c=1,\dots,C,\;
i=1,\dots,N_\ell
\Bigr),
\qquad \bx\in\CX.
\]
Thus, $\CH_{1:C}^{(\ell,n)}$
is an $\RE^{C \times N_\ell}$-valued stochastic process. By continuity of the
patch-extraction operators, its sample paths are continuous, and hence
$\CH_{1:C}^{(\ell,n)}$ can be regarded as a random element of 
$\CC(\CX,\RE^{C \times N_\ell})$, the space of continuous functions taking values in $\RE^{C \times N_\ell}$ with the sup norm.

The limiting matrix-valued covariance kernels are defined recursively as
follows. Set
\[
\CU^{(1)}(\bx)
:=
\Bigl(
h^{(1)}_{1,i}(\bx):
i=1,\dots,N_1
\Bigr),
\qquad \bx\in\CX.
\]
Recursively for $\ell=1,\dots,L$, given $\CU^{(\ell)}$, define 
a matrix-valued kernel $\CK^{(\ell+1)}$
and a new process $\CU^{(\ell+1)}$ by 
\begin{equation}
\label{eq:cnn_nngp_recursion}
\CK^{(\ell+1)}_{i,j}(\bx,\by)
:=
\frac{\mf_2^{(\ell)}}{M_\ell}
\sum_{m\in\CM_\ell}
\E\!\left[
T_m^{(i,\ell)}\bigl(\CU^{(\ell)}(\bx)\bigr)\,
T_m^{(j,\ell)}\bigl(\CU^{(\ell)}(\by)\bigr)
\right],
\end{equation}
for $i,j=1,\dots,N_{\ell+1}$ 
and $\bx,\by \in \CX$ and 
taking  $\CU^{(\ell+1)}$ to be a centered
$\RE^{N_{\ell+1}}$-valued Gaussian process with 
matrix-valued covariance  kernel
$\CK^{(\ell+1)}$, that is,
\[
\Cov\!\left(
\CU^{(\ell+1)}_i(\bx),
\CU^{(\ell+1)}_j(\by)
\right)
=
\CK^{(\ell+1)}_{i,j}(\bx,\by),
\qquad
\bx,\by\in\CX.
\]

The following theorem is the main result of the paper.
It shows that, under finite moment assumptions on the weights,
the Gaussian equivalence previously established for Gaussian priors
continues to hold, even at the functional level.
Since the weights are non-Gaussian, the first-layer preactivations are not
Gaussian in general. Indeed, at this layer there is no aggregation over a
number of channels tending to infinity.
For this reason, the next result only covers the layers $\ell\geq 2$.

\begin{theorem}[Gaussian equivalence for convolutional neural networks]
\label{thm:cnn_clt}
Let {\rm Assumptions~\ref{A1CNN-bis}--\ref{A2_0CNN-tris}--\ref{A3nongaussian}} hold with   the strengthened condition $q>\max(2,C_0N_0)$.
Then, for every hidden layer $\ell=2,\dots,L$ and every fixed
$C^*<\infty$ (or $\ell=L+1$ and $C^*=C_{L+1}$),
\[
\CH^{(\ell,n)}_{1:C^*}
\stackrel{\CL}{\longrightarrow}
\CH^{(\ell,\infty)}_{1:C^*}
\qquad\text{in  
$\,\, \mathcal{C}(\CX,\RE^{C^* \times N_\ell})$},
\]
where
\[
\CH^{(\ell,\infty)}_{1:C^*}(\bx)
=
\Bigl(
h^{(\ell,\infty)}_{c,i}(\bx):
c=1,\dots,C^*,\;
i=1,\dots,N_\ell
\Bigr),
\qquad \bx\in\CX,
\]
is a centered $\RE^{C^* \times N_\ell}$-valued Gaussian process with continuous trajectories. 
Its channels are independent and identically distributed, with each
channel having the same law as $\CU^{(\ell)}$. More precisely,
\[
\Cov\!\left(
h^{(\ell,\infty)}_{c,i}(\bx),
h^{(\ell,\infty)}_{c',j}(\by)
\right)
=
\delta_{c,c'}\,
\CK^{(\ell)}_{i,j}(\bx,\by),
\qquad
\bx,\by\in\CX,
\]
where the covariance kernels $\CK^{(\ell)}$ are defined recursively in
\eqref{eq:cnn_nngp_recursion}.
\end{theorem}

\subsection{Finite dimensional Gaussian convergence}

The functional Gaussian equivalence formulated in Theorem~\ref{thm:cnn_clt} relies on the convergence of finite-dimensional distributions together with tightness. While tightness requires stronger moment and regularity assumptions, finite-dimensional convergence alone can be established under weaker conditions.

\begin{assumption}[Polynomial growth assumption]
\label{A2_0CNN-bis}
The patch-extraction operators
$T^{(i,\ell)}$
are continuous.
Moreover, there exist
$r\ge1$
and
$A>0$
such that
\begin{equation}
\label{polyboundT}
|T^{(i,\ell)}_m(z)|
\le
A(1+\|z\|^r),
\qquad
z\in\RE^{N_\ell},
\end{equation}
for every $i,\ell,m$.
\end{assumption}

\begin{theorem}[Finite-dimensional convergence]
\label{thm:cnn_clt_FD}
Let Assumptions~\ref{A1CNN-bis}, \ref{A3nongaussian} and
\ref{A2_0CNN-bis} hold.
Fix $P\geq1$ and arbitrary points $\bx_1,\ldots,\bx_P\in\CX$.
Then, for every hidden layer $\ell=2,\ldots,L$ and every fixed
$C^*<\infty$ (or $\ell=L+1$ and $C^*=C_{L+1}$),
\[
\Bigl(
h^{(\ell)}_{c,i}(\bx_\mu):
c=1,\ldots,C^*,\
i=1,\ldots,N_\ell,\
\mu=1,\ldots,P
\Bigr)
\]
converges in distribution to a centered Gaussian random vector with
covariance matrix $\Sigma^{(\ell)}$ whose entries are
\begin{equation}\label{def.sigma.matrix}
\Sigma^{(\ell)}_{(c,i,\mu),(c',j,\nu)}
=
\delta_{c,c'}\,
\CK^{(\ell)}_{i,j}(\bx_\mu,\bx_\nu),
\end{equation}
for $c,c'=1,\ldots,C^*$, $i,j=1,\ldots,N_\ell$, and
$\mu,\nu=1,\ldots,P$, where the covariance kernels
$\CK^{(\ell)}$ are defined recursively in
\eqref{eq:cnn_nngp_recursion}.
\end{theorem}

The assumptions of Theorem~\ref{thm:cnn_clt_FD} are weaker than those required for functional convergence in Theorem~\ref{thm:cnn_clt} in two respects. First, Assumption~\ref{A2_0CNN-tris} implies Assumption~\ref{A2_0CNN-bis} with the same \(r\), while the converse is false. 
Second, finite-dimensional convergence only requires \(q>2\), whereas functional convergence requires \(q>C_0N_0\), that is, more than \(C_0N_0\) finite moments for the weights of the last
layer. In particular, when \(r=1\), \(q_\ell=q\) for every \(\ell\), and the moment assumption reduces to the existence of a finite \(q\)-th moment at each
layer.

\begin{remark}[Comparison with tensor program results]
Finite dimensional limits for neural networks with non-Gaussian
weights have been established by
\cite{YangGolikov} in the general framework of tensor programs,
which encompasses a broad class of neural network architectures,
including CNNs.

Their main result, Theorem 3.7, assumes the existence of moments
of all orders for the matrix weights and polynomial smoothness
of the nonlinearities. An alternative result, Theorem D.2 and
Corollary D.3 in their supplementary material, allows Lipschitz
nonlinearities but requires sub-Gaussian matrix weights.

In contrast, our Theorem~\ref{thm:cnn_clt_FD} establishes
finite-dimensional Gaussian convergence under finite moment
assumptions and mere continuity with polynomial growth of the
patch-extraction operators. Moreover, Theorem~\ref{thm:cnn_clt} establishes weak convergence
in the space of continuous functions on a compact input set,
still assuming only finitely many moments of the weights.

\end{remark}

\section{Proofs}
\label{sec:proofs}

The proof proceeds by induction over the network layers.
In Subsection~\ref{Sec:momentProp}, we establish the moment and increment bounds needed throughout the proof.
We then prove finite-dimensional convergence by the Cram\'er--Wold argument, studying the linear projections
\begin{equation}\label{lin.proj}
S_n^{(\ell)}(\alpha)
:=
\sum_{c=1}^{C^*}
\sum_{i=1}^{N_{\ell+1}}
\sum_{\mu=1}^P
\alpha_{c,i,\mu}
h^{(\ell+1)}_{c,i}(\bx_\mu),
\qquad
\alpha\in\RE^{C^*\times N_{\ell+1}\times P},
\end{equation}
for fixed arbitrary inputs $\bx_1,\dots,\bx_P\in\CX$.
The argument is divided into two steps: conditional CLT 
estimates  (Subsection \ref{Sec:Lyap}) and Gaussian approximation (Subsection \ref{Sec:GausLP}).
Finally, the increment bounds and the Kolmogorov--Chentsov criterion yield tightness in the space of continuous functions (Subsection \ref{Sec:Kolmogorov--Chentsov}).

%%%%%%%%%%%%%%%%%%%%%%%%%%%%%%%%%%%%%%%%%%%%%%%%%%%%%%%%%%%%%%%%%%%%%
% Moment propagation
%%%%%%%%%%%%%%%%%%%%%%%%%%%%%%%%%%%%%%%%%%%%%%%%%%%%%%%%%%%%%%%%%%%%%
\subsection{Moment and increment bounds}\label{Sec:momentProp}

We begin by establishing the propagation of moments through the network.
The key ingredient is Rosenthal's inequality.  An analogous strategy is used for the FCNN with general weights in \cite{Celli2026Wide}.

\begin{proposition}[Rosenthal inequality, \cite{rosenthal1970}]
Let $(X_i)_{i=1}^n$ be independent, centered random variables and let
$q\ge2$.
Then there exists a universal constant $K_q>0$ such that
\[
\E\Big|\sum_{i=1}^n X_i\Big|^q
\le
K_q\Big(
\sum_{i=1}^n\E|X_i|^q
+
\Big(\sum_{i=1}^n\E|X_i|^2\Big)^{q/2}
\Big).
\]
\end{proposition}

For future reference, let us introduce the following 
$\sigma$-field:
\[
\CF_\ell
:=
\sigma\!\left(
h^{(\ell)}_{c,j}(\bx):
c=1,\dots,C_\ell,\;
j=1,\dots,N_\ell,\;
\bx \in \CX
\right).
\]

\begin{proposition}[Propagation of moments through the layers]
\label{Prop:boundmom}
Let {\rm Assumptions~\ref{A1CNN-bis}--\ref{A3nongaussian}--\ref{A2_0CNN-bis}} hold.
For every $\ell=1,\dots,L+1$,
\begin{equation}\label{eq.increments.bound}
\sup_{n\ge1} \sup_{\bx \in \CX}
\E \!\left[ \|h^{(\ell)}_{1,:}(\bx)\|^{q_{\ell-1}} \right]
<\infty.
\end{equation}
\end{proposition}

\begin{proof}
Observe that
$q_0=qr^L$,
$q_{\ell-1}=rq_\ell$, $ \ell=1,\dots,L$.
We argue by induction on $\ell$.
For $\ell=1$, since $\CX$ is compact and the operators
$T^{(i,0)}$ are continuous,
$
\sup_{\bx\in\CX}
\bigl|T_m^{(i,0)}(x_{c',:})\bigr|<\infty$.
Since $C_0$, $M_0$ and $N_1$ are fixed and the weights have finite
$q_0$-moment by Assumption~\ref{A3nongaussian}, the claim follows
directly from the definition of the first layer.
Suppose now that \eqref{eq.increments.bound} holds. 
For fixed $c,i$,
\[
h^{(\ell+1)}_{c,i}(\bx)
=
\frac1{\sqrt{C_\ell}}
\sum_{c'=1}^{C_\ell}Y_{c'},
\]
where
\[
Y_{c'}=Y_{c'}(\bx,c,i)
=
\frac1{\sqrt{M_\ell}}
\sum_{m\in\CM_\ell}
W^{(\ell)}_{m,c',c}
T_m^{(i,\ell)}
\bigl(h^{(\ell)}_{c',:}(\bx)\bigr),
\qquad c'=1,\dots,C_\ell.
\]
The random variables $(Y_{c'})_{c'=1}^{C_\ell}$ are exchangeable and,
conditionally on $\CF_\ell$, independent and centered.
Applying Rosenthal's inequality conditionally on $\CF_\ell$ yields
\[
\E\!\left[
|h^{(\ell+1)}_{c,i}(\bx)|^{q_\ell}
\,\middle|\,
\CF_\ell
\right]
\le
\frac{K_{q_\ell}}{C_\ell^{q_\ell/2}}
\sum_{c'=1}^{C_\ell}
\E\!\left[
|Y_{c'}|^{q_\ell}
\,\middle|\,
\CF_\ell
\right]
+
K_{q_\ell}
\left(
\frac1{C_\ell}
\sum_{c'=1}^{C_\ell}
\E\!\left[
|Y_{c'}|^2
\,\middle|\,
\CF_\ell
\right]
\right)^{q_\ell/2}.
\]
Taking expectations,
$\E |h^{(\ell+1)}_{c,i}(\bx)|^{q_\ell}\le I_1+I_2$, where
\[
I_1
=
\frac{K_{q_\ell}}{C_\ell^{q_\ell/2}}
\sum_{c'=1}^{C_\ell}\E|Y_{c'}|^{q_\ell}
=
K_{q_\ell}C_\ell^{1-q_\ell/2}\E|Y_1|^{q_\ell},
\]
and, using Jensen's inequality twice,
\[
I_2
\le
\frac{K_{q_\ell}}{C_\ell}
\sum_{c'=1}^{C_\ell}
\E\left[
\left(
\E[|Y_{c'}|^2\mid\CF_\ell]
\right)^{q_\ell/2}
\right] 
\le
\frac{K_{q_\ell}}{C_\ell}
\sum_{c'=1}^{C_\ell}
\E|Y_{c'}|^{q_\ell}
=
K_{q_\ell}\E|Y_1|^{q_\ell}.
\]
Using again Jensen's inequality, together with
\eqref{polyboundT} and Assumption~\ref{A3nongaussian}, one has
\[
\E|Y_1|^{q_\ell}
\le
D_1\left(
1+\E\|h^{(\ell)}_{1,:}(\bx)\|^{rq_\ell}
\right),
\]
where $D_1$ is independent of $n$ and $C_\ell$.
Since $rq_\ell=q_{\ell-1}$,
the induction hypothesis implies
$\sup_{\bx\in\CX}\E|Y_1(\bx,c,i)|^{q_\ell}\le D_2$.

Combining the previous estimates,
$\sup_{\bx\in\CX}
\E|h^{(\ell+1)}_{c,i}(\bx)|^{q_\ell}
\le D_3$,
with $D_3$ independent of $n$, $i$, $c$ and $C_\ell$.
Since $N_{\ell+1}$ is fixed,
$
\sup_{\bx\in\CX}
\E\|h^{(\ell+1)}_{1,:}(\bx)\|^{q_\ell}
\le D_5$,
uniformly in $n$.
\end{proof}

\begin{proposition}[Propagation of increments]
\label{lem:increments}
Under  {\rm Assumptions~\ref{A1CNN-bis}--\ref{A2_0CNN-tris}--\ref{A3nongaussian}} with the strengthened condition $q>\max(2,C_0N_0)$, for every $\ell=1,\dots,L+1$ there
is a constant $B_\ell<\infty$, independent of $n$, such that
\begin{equation}
\label{eq:increment}
\sup_{n\ge1}\,
\E\Bigl[\bigl\|h^{(\ell)}_{1,:}(\bx)-h^{(\ell)}_{1,:}(\by)\bigr\|^{q_{\ell-1}}\Bigr]
\le
B_\ell\,\|\bx-\by\|^{q_{\ell-1}},
\qquad \bx,\by\in\CX .
\end{equation}
\end{proposition}

\begin{proof}
We argue by induction on $\ell$.
For $\ell=1$, by Assumption~\ref{A2_0CNN-tris} and the compactness of
$\CX$, there exists $A^*<\infty$ such that, for every $i,m,c'$ and
$\bx,\by\in\CX$,
$\left|
T_m^{(i,0)}(x_{c',:})-T_m^{(i,0)}(y_{c',:})
\right|
\le A^*\|\bx-\by\|$.
Since $C_0$, $M_0$ and $N_1$ are fixed and the weights have finite
$q_0$-moment, the definition of the first layer gives
\[
\E\left[
\|h^{(1)}_{1,:}(\bx)-h^{(1)}_{1,:}(\by)\|^{q_0}
\right]
\le B_1\|\bx-\by\|^{q_0}.
\]

Suppose now that the claim holds at layer $\ell$ and set $p=q_\ell$.
For fixed $c,i$ and $\bx,\by\in\CX$, write
\[
h^{(\ell+1)}_{c,i}(\bx)-h^{(\ell+1)}_{c,i}(\by)
=
\frac1{\sqrt{C_\ell}}
\sum_{c'=1}^{C_\ell}\widetilde Y_{c'},
\]
where
\[
\widetilde Y_{c'}
:=
\frac1{\sqrt{M_\ell}}
\sum_{m\in\CM_\ell}
W^{(\ell)}_{m,c',c}\,
\widetilde T_{m,c'}(\bx,\by),
%\]
%with
%\[
\quad \widetilde T_{m,c'}(\bx,\by)
:=
T_m^{(i,\ell)}
\bigl(h^{(\ell)}_{c',:}(\bx)\bigr)
-
T_m^{(i,\ell)}
\bigl(h^{(\ell)}_{c',:}(\by)\bigr).
\]
Arguing as in the proof of Proposition~\ref{Prop:boundmom}, with
$Y_{c'}$ replaced by $\widetilde Y_{c'}$, conditional Rosenthal's
inequality and Jensen's inequality yield
\[
\E\left|
h^{(\ell+1)}_{c,i}(\bx)
-h^{(\ell+1)}_{c,i}(\by)
\right|^p
\le D_1\E|\widetilde Y_1|^p
\le
D_2\sum_{m\in\CM_\ell}
\E|\widetilde T_{m,1}(\bx,\by)|^p .
\]

By Assumption~\ref{A2_0CNN-tris},
\[
|\widetilde T_{m,1}(\bx,\by)|
\le
A\,G(\bx,\by)\Delta(\bx,\by),
\]
where
\[
G(\bx,\by)
=
1+\|h^{(\ell)}_{1,:}(\bx)\|^{r-1}
 +\|h^{(\ell)}_{1,:}(\by)\|^{r-1},
\qquad
\Delta(\bx,\by)
=
\|h^{(\ell)}_{1,:}(\bx)-h^{(\ell)}_{1,:}(\by)\|.
\]
If $r=1$, then $G\equiv 3$ and, since $p\le q_{\ell-1}$, the induction
hypothesis together with the monotonicity of $L^s$ norms gives
$\E|\widetilde T_{m,1}(\bx,\by)|^p\le D_3\|\bx-\by\|^p$.
If $r>1$, H\"older's inequality with exponents $r/(r-1)$ and $r$ yields
\[
\E|\widetilde T_{m,1}(\bx,\by)|^p
\le
A^p
\bigl(\E G(\bx,\by)^{pr/(r-1)}\bigr)^{(r-1)/r}
\bigl(\E\Delta(\bx,\by)^{pr}\bigr)^{1/r}.
\]
Since
$G(\bx,\by)^{pr/(r-1)}
\le D\bigl(1+\|h^{(\ell)}_{1,:}(\bx)\|^{pr}+\|h^{(\ell)}_{1,:}(\by)\|^{pr}\bigr)$
and $pr=q_{\ell-1}$,Proposition~\ref{Prop:boundmom} bounds the
first factor uniformly in $n,\bx,\by$, while the induction hypothesis,
together with the monotonicity of $L^s$ norms, gives
\[
\bigl(\E\Delta(\bx,\by)^{pr}\bigr)^{1/r}
\le D_3\|\bx-\by\|^p.
\]
Consequently,
\[
\E\left|
h^{(\ell+1)}_{c,i}(\bx)
-h^{(\ell+1)}_{c,i}(\by)
\right|^p
\le D_4\|\bx-\by\|^p.
\]
Since $N_{\ell+1}$ is fixed, this yields
\[
\E\left[
\|h^{(\ell+1)}_{1,:}(\bx)
-h^{(\ell+1)}_{1,:}(\by)\|^{q_\ell}
\right]
\le B_{\ell+1}\|\bx-\by\|^{q_\ell},
\]
uniformly in $n$, and the induction is complete.
\end{proof}

\subsection{Conditional Lyapunov condition and variances}\label{Sec:Lyap}

Fix $P \geq 1$ and $\bx_1,\dots,\bx_P$ in $\CX$. 
Using \eqref{convolutiona.recur1}, the linear projections
\eqref{lin.proj} can be written as
\[
S_n^{(\ell)}(\alpha) = \frac1{\sqrt{C_\ell}} \sum_{c'=1}^{C_\ell} \xi_{c'}^{(n,\ell)}(\alpha),
\]
where
\[
\xi_{c'}^{(n,\ell)}(\alpha) := \frac1{\sqrt{M_\ell}} \sum_{m\in\CM_\ell} \sum_{c=1}^{C^*} \sum_{i=1}^{N_{\ell+1}} \sum_{\mu=1}^P \alpha_{c,i,\mu}\, W^{(\ell)}_{m,c',c}\, T_m^{(i,\ell)} \bigl( h^{(\ell)}_{c',:}(\bx_\mu) \bigr).
\]

Conditionally on $\CF_\ell$, the random variables
$\bigl( \xi_{c'}^{(n,\ell)}(\alpha) \bigr)_{c'=1}^{C_\ell}$
are independent and centered. The next lemma verifies the corresponding
conditional Lyapunov condition.

\begin{lemma}[Conditional Lyapunov condition]
\label{lemma:lyapunov_cnn}
Under  {\rm Assumptions~\ref{A1CNN-bis}--\ref{A3nongaussian}--\ref{A2_0CNN-bis}}, for every $\ell=1,\ldots,L$, $\alpha\in\RE^{C^*\times N_{\ell+1}\times P}$,
and $p\in(2,q]$,
\[
\frac{1}{C_\ell^{p/2}}
\sum_{c'=1}^{C_\ell}
\E\!\left[
|\xi_{c'}^{(n,\ell)}(\alpha)|^p
\,\middle|\,\CF_\ell
\right]
\longrightarrow 0
\qquad\text{in $L^1$}.
\]
\end{lemma}

\begin{proof}
Since the random variable above is non-negative, it is enough to prove that
\[
\E\!\left[
\frac1{C_\ell^{p/2}}
\sum_{c'=1}^{C_\ell}
\E\!\left[
|\xi_{c'}^{(n,\ell)}(\alpha)|^p
\,\middle|\,\CF_\ell
\right]
\right]
\longrightarrow 0.
\]
By exchangeability,
\[
\E\!\left[
\frac1{C_\ell^{p/2}}
\sum_{c'=1}^{C_\ell}
\E\!\left[
|\xi_{c'}^{(n,\ell)}(\alpha)|^p
\,\middle|\,\CF_\ell
\right]
\right]
=
C_\ell^{1-p/2}
\E\!\left[
|\xi_{1}^{(n,\ell)}(\alpha)|^p
\right].
\]
Using standard inequalities, \eqref{polyboundT},
Assumption~\ref{A3nongaussian}, and the fact that
$M_\ell$, $C^*$, $N_{\ell+1}$ and $P$ are fixed,
\begin{align*}
\E\!\left[
|\xi_{1}^{(n,\ell)}(\alpha)|^p
\right]
&\le
D_1
\sum_{m\in\CM_\ell}
\sum_{c=1}^{C^*}
\sum_{i=1}^{N_{\ell+1}}
\sum_{\mu=1}^P
|\alpha_{c,i,\mu}|^p
\E\left[
\bigl(
1+|h^{(\ell)}_{1,:}(\bx_\mu)|^r
\bigr)^p
\right] 
\\
&\le
D_2
\sum_{\mu=1}^P
\left(
1+
\E\|h^{(\ell)}_{1,:}(\bx_\mu)\|^{rp}
\right).
\end{align*}
Since
$rp\le rq =q_{L-1}\le q_{\ell-1}$,
the monotonicity of $L^s$ norms and Proposition~\ref{Prop:boundmom}
imply
\[
\sup_{n\ge1}\sup_{\mu=1,\ldots,P}
\E\!\left[
\|h^{(\ell)}_{1,:}(\bx_\mu)\|^{rp}
\right]
<\infty.
\]
Hence
$\E [|\xi_{1}^{(n,\ell)}(\alpha)|^p]
\le D_3$
uniformly in $n$. Since $C_\ell\longrightarrow\infty$ and
$1-p/2<0$, we conclude that
\[
C_\ell^{1-p/2}
\E\!\left[
|\xi_{1}^{(n,\ell)}(\alpha)|^p
\right]
\le
D_3 C_\ell^{1-p/2}
\longrightarrow0,
\]
\end{proof}

%%%%%%%%%%%%%%%%%%%%%%%%%%%%%%%%%%%%%%%%%%%%%%%%%%%%%%%%%%%%%%%%%%%%%
% Conditional variances
%%%%%%%%%%%%%%%%%%%%%%%%%%%%%%%%%%%%%%%%%%%%%%%%%%%%%%%%%%%%%%%%%%%%%

%\subsection{Conditional variances}\label{Sec:ConVar}
We next identify the asymptotic behavior of the conditional variances.
Conditionally on $\CF_\ell$, the weights
$\{W^{(\ell)}_{m,c',c}:m\in\CM_\ell,\ c=1,\dots,C^*\}$
are independent, centered, and satisfy
$\E\!\left[ (W^{(\ell)}_{m,c',c})^2 \right] = \mf_2^{(\ell)}$.
Therefore,
\begin{align}
\E\!\left[
|\xi_{c'}^{(n,\ell)}(\alpha)|^2
\,\middle|\,
\CF_\ell
\right]
&=
\frac{\mf_2^{(\ell)}}{M_\ell}
\sum_{m\in\CM_\ell}
\sum_{c=1}^{C^*}
\left(
\sum_{i=1}^{N_{\ell+1}}
\sum_{\mu=1}^{P}
\alpha_{c,i,\mu}
T_m^{(i,\ell)}
\bigl(
h^{(\ell)}_{c',:}(\bx_\mu)
\bigr)
\right)^2.
\label{eq:conditional_variance_xi}
\end{align}

The conditional variance of $S_n^{(\ell)}(\alpha)$ is therefore
\begin{equation}
\label{eq:conditional_variance_sum}
V_n^{(\ell)}(\alpha)
:=
\frac1{C_\ell}
\sum_{c'=1}^{C_\ell}
\E\!\left[
|\xi_{c'}^{(n,\ell)}(\alpha)|^2
\,\middle|\,
\CF_\ell
\right].
\end{equation}

The key observation is that
$V_n^{(\ell)}(\alpha)$
can be expressed as an empirical average of a deterministic function
evaluated at the random matrices
$U_c^{(\ell,n)}$,  where
\begin{equation}\label{defU}
U_c^{(\ell,n)}
:=
\bigl[
h^{(\ell)}_{c,i}(\bx_\mu):
i=1,\dots,N_\ell,\;
\mu=1,\dots,P
\bigr]
\in\RE^{N_\ell\times P},
\qquad
c=1,\dots,C_\ell .
\end{equation}
By construction, the network weights are independent and identically
distributed across channels. Consequently, the family
$\bigl(U_c^{(\ell,n)}\bigr)_{c=1}^{C_\ell}$
is exchangeable.

For
$u=(u_{j,\mu}) \in \mathbb R^{N_\ell\times P}$,
let $u_{:,\mu}$ denote the $\mu$-th column of $u$, and define
\[
\Phi_{\ell,\alpha}(u) := \frac{\mf_2^{(\ell)}}{M_\ell} \sum_{m\in\CM_\ell} \sum_{c=1}^{C^*} \left( \sum_{i=1}^{N_{\ell+1}} \sum_{\mu=1}^{P} \alpha_{c,i,\mu} T_m^{(i,\ell)}(u_{:,\mu}) \right)^2.
\]
Hence,
$V_n^{(\ell)}(\alpha) = \frac1{C_\ell} \sum_{c'=1}^{C_\ell} \Phi_{\ell,\alpha} \bigl( U_{c'}^{(\ell,n)} \bigr)$.
The convergence of
$V_n^{(\ell)}(\alpha)$
is therefore reduced to a law of large numbers for exchangeable triangular
arrays. 
%The following auxiliary result is  proved in Section~\ref{Appendix}. 

\begin{lemma}
\label{lemma:empirical_average_exchangeable}
Let $(Z_{n,k})_{k=1,\dots,\kappa_n}$ be an array of exchangeable
$\mathbb R^D$-valued random variables, and let
$\Psi:\mathbb R^D\to\mathbb R$ be a continuous function.
Assume that $\kappa_n\to\infty$ and that

\begin{itemize}
\item[(i)] for every fixed $C\ge1$,
$(Z_{n,1},\dots,Z_{n,C}) \stackrel{\CL}{\longrightarrow} (Z_1,\dots,Z_C)$,
where $(Z_k)_{k\ge1}$ are i.i.d.;

\item[(ii)] there exists $\varepsilon>0$ such that
$\sup_{n\ge1} \E\!\left[ |\Psi(Z_{n,1})|^{1+\varepsilon} \right] <\infty$.
\end{itemize}

Then
$\frac1{\kappa_n} \sum_{k=1}^{\kappa_n} \Psi(Z_{n,k}) \stackrel{\PP}{\longrightarrow} \E\!\left[ \Psi(Z_1) \right]$.

\end{lemma}

\begin{proof}%[Proof of Lemma \ref{lemma:empirical_average_exchangeable}]
Without real loss of generality we assume $\kappa_n=n$. 
Define $Z_{n,k}=Z_{n,n}$ if $k \geq n$. 
First of all note that 
$(Z_{n,k})_{k\ge1}\stackrel{\CL}{\to} (Z_k)_{k\ge1}$ 
in $(\mathbb{R}^D)^{\mathbb N}$.
Indeed, by assumption, the finite-dimensional distributions converge and 
tightness in $(\mathbb{R}^D)^{\mathbb N}$ can be proved by a standard product of compacts argument. 
Hence, since $(\mathbb{R}^D)^{\mathbb N}$ is Polish, by Skorokhod's representation theorem, possibly changing the probability space, we may assume that
\[
(Z_{n,k})_{k\ge1}\to (Z_k)_{k\ge1}
\qquad\text{almost surely in }(\mathbb{R}^D)^{\mathbb N}.
\]
In particular, for every fixed $k\ge1$,
$Z_{n,k}\to Z_k$ almost surely.
Since $\Psi$ is continuous and $Z_{n,1}\to Z_1$ almost surely, one has
$\Psi(Z_{n,1}) \to \Psi(Z_1)$ almost surely.
Fatou's lemma and (A2) yield
\begin{equation}\label{fatou}
\E\bigl[|\Psi(Z_1)|^{1+\varepsilon}\bigr]
\le
\liminf_{n\to\infty}
\E\bigl[|\Psi(Z_{n,1})|^{1+\varepsilon}\bigr]
<+\infty.
\end{equation}
We decompose
$\frac1n\sum_{k=1}^n \Psi(Z_{n,k}) = \frac1n\sum_{k=1}^n \Psi(Z_k)+\Delta_n$,
where
$\Delta_n := \frac1n\sum_{k=1}^n\bigl(\Psi(Z_{n,k})-\Psi(Z_k)\bigr)$.
Since $(Z_k)_{k\ge1}$ are iid
and \eqref{fatou} holds, the law of large numbers gives
\[
\frac1n\sum_{k=1}^n \Psi(Z_k)
\stackrel{\PP}{\to}
\E[\Psi(Z_1)].
\]
It remains to prove that $\Delta_n\to0$ in probability. By the triangle inequality and exchangeability,
\[
\E|\Delta_n|
\le
\frac1n\sum_{k=1}^n \E\bigl[|\Psi(Z_{n,k})-\Psi(Z_k)|\bigr]
=
\E\bigl[|\Psi(Z_{n,1})-\Psi(Z_1)|\bigr].
\]

Moreover,
$\E|\Psi(Z_{n,1})-\Psi(Z_1)|^{1+\varepsilon} \le 2^\varepsilon \Big (\E \big[|\Psi(Z_{n,1})|^{1+\varepsilon}\big]+\E\big [|\Psi(Z_1)|^{1+\varepsilon}\big]\Big)$,
so by  \eqref{fatou} the sequence
$\bigl(|\Psi(Z_{n,1})-\Psi(Z_1)|\bigr)_{n\ge1}$
is uniformly integrable. Therefore
$\E\bigl[|\Psi(Z_{n,1})-\Psi(Z_1)|\bigr]\to0$,
and hence
$\E|\Delta_n|\to0$.
This implies $\Delta_n\to0$ in $L^1$, and therefore in probability. The claim follows.
\end{proof}

The next lemma identifies the limit of the conditional variances under the
induction hypothesis.

\begin{lemma}
\label{lemma:variance_limit}
Under  {\rm Assumptions~\ref{A1CNN-bis}--\ref{A3nongaussian}--\ref{A2_0CNN-bis}},  
fix $\ell=1,\dots,L$ and assume the induction hypothesis
\begin{equation}
\label{indH}
\text{for every }C\ge1,\qquad
(U_1^{(\ell,n)},\dots,U_C^{(\ell,n)})
\stackrel{\CL}{\longrightarrow}
(U_1^{(\ell)},\dots,U_C^{(\ell)}),
\end{equation}
where $(U_k^{(\ell)})_{k\ge1}$ are i.i.d.
Then, for every fixed $C^*\ge1$ and every
$\alpha\in\mathbb R^{C^*\times N_{\ell+1}\times P}$,
$V_n^{(\ell)}(\alpha) \stackrel{\PP}{\longrightarrow} \sigma_\ell^2(\alpha) := \E\!\left[ \Phi_{\ell,\alpha}(U_1^{(\ell)}) \right]$.
\end{lemma}

\begin{proof}

Since each $T_m^{(i,\ell)}$ is continuous and all sums are finite,
$\Phi_{\ell,\alpha}:\mathbb R^{N_\ell\times P} \longrightarrow \mathbb R$
is continuous.
We apply Lemma~\ref{lemma:empirical_average_exchangeable} with
$Z_{n,k}=U_k^{(\ell,n)}$,  $\Psi=\Phi_{\ell,\alpha}$ and $\varepsilon=\frac q2-1>0$.
By \eqref{polyboundT}, there exists a constant $D_1>0$ such that
\[
|\Phi_{\ell,\alpha}(u)|
\le
D_1
\left(
1+\sum_{\mu=1}^P
\|u_{:,\mu}\|^{2r}
\right),
\qquad
u\in\mathbb R^{N_\ell\times P}.
\]
Hence
\[
|\Phi_{\ell,\alpha}(U_1^{(\ell,n)})|^{1+\varepsilon}
\le
D_2
\left(
1+
\sum_{\mu=1}^P
\|h^{(\ell)}_{1,:}(\bx_\mu)\|^{2r(1+\varepsilon)}
\right).
\]
Since
$2r(1+\varepsilon) =rq \leq r q =q_{L-1} \leq q_{\ell-1}$.
Proposition~\ref{Prop:boundmom} gives
\[
\sup_{n\ge1}
\E\!\left[
|\Phi_{\ell,\alpha}(U_1^{(\ell,n)})|^{1+\varepsilon}
\right]
<\infty.
\]
Thus, Lemma~\ref{lemma:empirical_average_exchangeable} and the induction hypothesis \eqref{indH} gives 
\[
V_n^{(\ell)}(\alpha)
=
\frac1{C_\ell}
\sum_{c'=1}^{C_\ell}
\Phi_{\ell,\alpha}
\bigl(
U_{c'}^{(\ell,n)}
\bigr)
\stackrel{\PP}{\longrightarrow}
\E\!\left[
\Phi_{\ell,\alpha}(U_1^{(\ell)})
\right],
\]
which is the desired conclusion.

\end{proof}
%%%%%%%%%%%%%%%%%%%%%%%%%%%%%%%%%%%%%%%%%%%%%%%%%%%%%%%%%%%%%%%
% Lyapunov CLT
%%%%%%%%%%%%%%%%%%%%%%%%%%%%%%%%%%%%%%%%%%%%%%%%%%%%%%%%%%%%%%%
%%%%%%%%%%%%%%%%%%%%%%%%%%%%%%%%%%%%%%%%%%%%%%%%%%%%%%%%%%%%%%%%%%%%%
% Gaussian limit for the linear projections
%%%%%%%%%%%%%%%%%%%%%%%%%%%%%%%%%%%%%%%%%%%%%%%%%%%%%%%%%%%%%%%%%%%%%

\subsection{Gaussian approximation and finite-dimensional convergence}\label{Sec:GausLP}

We shall use the following quantitative version of the Lyapunov central
limit theorem.

\begin{lemma}[Lyapunov central limit theorem]
\label{Lyapunovlemma}
Let $Z_1,\dots,Z_n$ be independent centered random variables such that
$\E|Z_j|^{2+\delta}<\infty$ for some $\delta\in(0,1]$. Define
$v_n = \sum_{j=1}^n \Var(Z_j)$
and let
$\phi_n(t) = \E\!\left[ e^{it\sum_{j=1}^n Z_j} \right]$
be the characteristic function of $\sum_{j=1}^n Z_j$.
Then
\[
\left|
\phi_n(t)-\phi_0(\sqrt{v_n}\,t)
\right|
\le
D|t|^{2+\delta}
\sum_{j=1}^n
\E|Z_j|^{2+\delta},
\]
where
$\phi_0(t)=e^{-t^2/2}$
and $D>0$ is an absolute constant.
\end{lemma}

The proof is standard and it is omitted. 
See e.g.\ Section 8.4 in \cite{ChowTeicher}.  

The next lemma establishes the central limit theorem for the linear
projections.

\begin{lemma}
\label{lemma:conditional_clt_projection}
Under  {\rm Assumptions~\ref{A1CNN-bis}--\ref{A3nongaussian}--\ref{A2_0CNN-bis}}, fix $\ell=1,\dots,L$ and assume the induction hypothesis
\eqref{indH}. Then, for every
$\alpha\in\mathbb R^{C^*\times N_{\ell+1}\times P}$,
one has
\[
S_n^{(\ell)}(\alpha) = \frac1{\sqrt{C_\ell}} \sum_{c'=1}^{C_\ell} \xi_{c'}^{(n,\ell)}(\alpha) \stackrel{\CL}{\longrightarrow} \CN\!\left( 0,\sigma_\ell^2(\alpha) \right).
\]
\end{lemma}

\begin{proof}
Conditionally on $\CF_\ell$, the random variables
$ \xi_{c'}^{(n,\ell)}(\alpha)$,  ${c'=1,\dots,C_\ell}$,
are independent and centered. Define
$Z_{n,c'} = \frac1{\sqrt{C_\ell}} \xi_{c'}^{(n,\ell)}(\alpha)$,
so that
$S_n^{(\ell)}(\alpha) = \sum_{c'=1}^{C_\ell} Z_{n,c'}$.
Let
$\phi_n(t) = \E[ e^{itS_n^{(\ell)}(\alpha)} | \CF_\ell ]$
denote the conditional characteristic function.
Fix $p\in(2,\min\{q,3\}]$
and write $p=2+\delta$, with $\delta\in(0,1]$.
Applying Lemma~\ref{Lyapunovlemma} conditionally on $\CF_\ell$ gives
$$
\left| \phi_n(t) - \phi_0(\sqrt{v_n}\,t) \right| \le D|t|^p \sum_{c'=1}^{C_\ell} \E\!\left[ |Z_{n,c'}|^p \,\middle|\, \CF_\ell \right],
$$
where
$v_n = \sum_{c'=1}^{C_\ell} \Var(Z_{n,c'}\mid\CF_\ell) = V_n^{(\ell)}(\alpha)$.
Hence
\[
\left|
\phi_n(t)
-
\phi_0\!\left(
\sqrt{V_n^{(\ell)}(\alpha)}\,t
\right)
\right|
\le
D|t|^p
\frac1{C_\ell^{p/2}}
\sum_{c'=1}^{C_\ell}
\E\!\left[
|\xi_{c'}^{(n,\ell)}(\alpha)|^p
\,\middle|\,
\CF_\ell
\right].
\]
By Lemma~\ref{lemma:lyapunov_cnn},
the left hand side converges to zero in $L^1$.
Therefore
\[
\left|
\E\!\left[
e^{itS_n^{(\ell)}(\alpha)}
\right]
-
\E\!\left[
\phi_0\!\left(
\sqrt{V_n^{(\ell)}(\alpha)}\,t
\right)
\right]
\right|
\longrightarrow
0.
\]
By Lemma~\ref{lemma:variance_limit},
$V_n^{(\ell)}(\alpha) \stackrel{\PP}{\longrightarrow} \sigma_\ell^2(\alpha)$.
Since
$0 \le \phi_0(\sqrt{x},t) = e^{-t^2x/2} \le 1$,
the continuous mapping theorem and dominated convergence yield
\[
\E\left[ \phi_0\left( \sqrt{V_n^{(\ell)}(\alpha)} \, t \right) \right] \to \exp\left( -\frac12t^2\sigma_\ell^2(\alpha) \right).
\]
Consequently,
$\E[ e^{itS_n^{(\ell)}(\alpha)}] \to \exp ( -\frac12t^2\sigma_\ell^2(\alpha) )$, which gives the thesis. 
%which is the characteristic function of
%$\CN(0,\sigma_\ell^2(\alpha))$.
\end{proof}

We now use the convergence of the linear projections to complete the induction step for the finite-dimensional distributions.

\begin{proof}[Proof of Theorem \ref{thm:cnn_clt_FD}]
At the first layer, the random matrices $U_c^{(1,n)}$ defined in \eqref{defU} do not depend on $n$ and are i.i.d.\ across $c$, since distinct channels depend on independent families of
weights. Hence, for every fixed $C\geq1$,
\[
\bigl(
U_1^{(1,n)},\ldots,U_C^{(1,n)}
\bigr)
\stackrel{\CL}{\longrightarrow}
\bigl(
U_1^{(1)},\ldots,U_C^{(1)}
\bigr),
\]
where $(U_c^{(1)})_{c\geq1}$ are i.i.d.\ copies of $U_1^{(1,n)}$. Thus \eqref{indH} holds for $\ell=1$.

Now fix $\ell=1,\ldots,L$, assume the induction hypothesis \eqref{indH}, and fix $C^*<\infty$ if $\ell<L$
(or $C^*=C_{L+1}$ if $\ell=L$).
For every
$\alpha\in\mathbb R^{C^*\times N_{\ell+1}\times P}$,
\[
S_n^{(\ell)}(\alpha)
=
\sum_{c=1}^{C^*}
\sum_{i=1}^{N_{\ell+1}}
\sum_{\mu=1}^P
\alpha_{c,i,\mu}
\bigl(U_c^{(\ell+1,n)}\bigr)_{i,\mu}.
\]
By Lemma~\ref{lemma:conditional_clt_projection},
\[
S_n^{(\ell)}(\alpha)
\stackrel{\CL}{\longrightarrow}
\CN\bigl(0,\sigma_\ell^2(\alpha)\bigr).
\]
It remains to identify the covariance associated with this limiting
quadratic form. Expanding the square in $\Phi_{\ell,\alpha}$ gives
\[
\sigma_\ell^2(\alpha)
=
\E\bigl[\Phi_{\ell,\alpha}(U_1^{(\ell)})\bigr]
=
\sum_{c=1}^{C^*}
\sum_{i,j=1}^{N_{\ell+1}}
\sum_{\mu,\nu=1}^P
\alpha_{c,i,\mu}\alpha_{c,j,\nu}\,
\CK^{(\ell+1)}_{i,j}(\bx_\mu,\bx_\nu)
\]
which is precisely the quadratic form associated with the covariance matrix $\Sigma^{(\ell+1)}$ defined in \eqref{def.sigma.matrix}.
The Cramér--Wold theorem yields that
$\bigl(U_1^{(\ell+1,n)},\ldots,U_{C^*}^{(\ell+1,n)}\bigr)$
converges in distribution to a centered Gaussian vector with covariance matrix $\Sigma^{(\ell+1)}$. Since $\Sigma^{(\ell+1)}$ is block-diagonal across channels, the limiting matrices
$U_1^{(\ell+1)},\ldots,U_{C^*}^{(\ell+1)}$
are independent and identically distributed, so that \eqref{indH} holds at layer $\ell+1$.

By induction, \eqref{indH} holds for every $\ell=1,\ldots,L$, which gives the claimed Gaussian convergence at the hidden layers $\ell=2,\ldots,L$. Finally,
applying the above induction step with $\ell=L$ and $C^*=C_{L+1}$ gives the claimed convergence at the output layer $\ell=L+1$.
\end{proof}

\subsection{Functional convergence}\label{Sec:Kolmogorov--Chentsov}

We conclude the proof of Theorem~\ref{thm:cnn_clt} by combining the
finite-dimensional convergence of  Theorem \ref{thm:cnn_clt_FD} with the
increment bounds of Proposition~\ref{lem:increments}.

\begin{proof}[Proof of Theorem~\ref{thm:cnn_clt}]
Fix an admissible layer $\ell$ and $C^*$, and set
$d=C_0N_0$, $p=q_{\ell-1}$.
By Proposition~\ref{lem:increments}, since $C^*$ is fixed, there exists
$B<+\infty$ such that
\[
\sup_{n\geq1}
\E\left[
\left\|
\CH_{1:C^*}^{(\ell,n)}(\bx)
-
\CH_{1:C^*}^{(\ell,n)}(\by)
\right\|^p
\right]
\leq
B\|\bx-\by\|^p,
\qquad \bx,\by\in\CX \subset \RE^d.
\]
%Moreover,
%$p=q_{\ell-1}\geq q>d$.

For every fixed $\bx\in\CX$, Theorem~\ref{thm:cnn_clt_FD}, with $P=1$,
implies that the sequence
$\bigl(\CH_{1:C^*}^{(\ell,n)}(\bx)\bigr)_{n\geq1}$ is tight in
$\RE^{C^* \times N_\ell}$.

If $\CX$ were a compact rectangle, the above increment estimate,
together with $p=q_{\ell-1}\geq q>d$, would yield tightness by the classical
Kolmogorov--Chentsov criterion. Since $\CX$ is allowed here to be an
arbitrary compact subset of $\RE^d$, we use instead the more general
tightness criterion in Proposition~4.1 of \cite{KratschmerUrusov2023}.  It follows that the laws of
$\CH_{1:C^*}^{(\ell,n)}$ are tight in
$C(\CX,\RE^{C^* \times N_\ell})$.

Combining this tightness result with the finite-dimensional convergence
established in Theorem~\ref{thm:cnn_clt_FD}, Lemma 16.2 \cite{Kallenberg} yields
$
\CH_{1:C^*}^{(\ell,n)}
\stackrel{\CL}{\longrightarrow}
\CH_{1:C^*}^{(\ell,\infty)}$ in $
C(\CX,\RE^{C^* \times N_\ell}),
$
which concludes the proof.
\end{proof}

%\bibliographystyle{amsplain}
%\bibliography{biblioCNN}

\providecommand{\bysame}{\leavevmode\hbox to3em{\hrulefill}\thinspace}
\providecommand{\MR}{\relax\ifhmode\unskip\space\fi MR }
\providecommand{\MRhref}[2]{%
  \href{http://www.ams.org/mathscinet-getitem?mr=#1}{#2}
}
\providecommand{\href}[2]{#2}

\end{document}